\documentclass{amsart}
\usepackage{latexsym,amsxtra,amscd,ifthen, amsmath,color,url}
\usepackage[usenames,dvipsnames]{xcolor} \usepackage{amsfonts}
\usepackage{verbatim} \usepackage{amsmath} \usepackage{amsthm}
\usepackage{amssymb}
\usepackage[linkcolor=red,citecolor=blue,urlcolor=blue]{hyperref}
\usepackage[notcite,notref,final]{showkeys}

\usepackage{pgf} \usepackage{tikz} \usetikzlibrary{arrows,automata}
\usepackage[latin1]{inputenc} \usepackage{verbatim}

\numberwithin{equation}{section} \theoremstyle{plain}
\newtheorem{theorem}{Theorem}[section]
\newtheorem{lemma}[theorem]{Lemma}

\newtheorem{proposition}[theorem]{Proposition}
\newtheorem{hypothesis}[theorem]{Hypothesis}
\newtheorem{stand-hypothesis}[theorem]{Standing Hypothesis}
\newtheorem{corollary}[theorem]{Corollary}
\newtheorem{conjecture}[theorem]{Conjecture}

 \theoremstyle{definition}
\newtheorem{definition}[theorem]{Definition}
\newtheorem{example}[theorem]{Example}

\newtheorem{remark}[theorem]{Remark}

\mathchardef\mhyphen="2D

\makeatletter              
\let\c@equation\c@theorem  
\makeatother

\newcommand{\e}{\mathrm e}

\DeclareMathOperator{\coker}{coker}

\DeclareMathOperator{\hdet}{hdet} 
 
\DeclareMathOperator{\gldim}{{\sf gldim}}
\DeclareMathOperator{\repdim}{{\sf repdim}} 

\DeclareMathOperator{\uHom}{\underline{Hom}}
\DeclareMathOperator{\uEnd}{\underline{End}}
\DeclareMathOperator{\uExt}{\underline{Ext}}
\DeclareMathOperator{\uTor}{\underline{Tor}}
\DeclareMathOperator{\uotimes}{\underline{\otimes}}
\DeclareMathOperator{\Ext}{Ext} \DeclareMathOperator{\Tor}{Tor}

 \DeclareMathOperator{\injdim}{{\sf injdim}}
 
\DeclareMathOperator{\End}{End} 
 \DeclareMathOperator{\D}{{\sf D}}
\DeclareMathOperator{\Hom}{Hom} \DeclareMathOperator{\RHom}{RHom}
\DeclareMathOperator{\Grm}{{\sf Grm}} 
\DeclareMathOperator{\grm}{{\sf grm}} 
\DeclareMathOperator{\Prm}{{\sf Prm}}
\DeclareMathOperator{\prm}{{\sf prm}}
\DeclareMathOperator{\Grprm}{{\sf Grprm}}
\DeclareMathOperator{\grprm}{{\sf grprm}}
 
\DeclareMathOperator{\dep}{{\sf dep}} 
\DeclareMathOperator{\cd}{{\sf cd}}

\DeclareMathOperator{\Mod}{{\sf Mod}} 
 
\DeclareMathOperator{\Xyz}{{\sf Xyz}} \DeclareMathOperator{\xyz}{{\sf
xyz}} 
\DeclareMathOperator{\add}{{\sf add}}

\newcommand{\fm}{\mathfrak{m}} 
 
\newcommand{\op}{\textup{op}}

\newcommand{\kk}{\Bbbk}

\newtheorem{introthm}{Theorem}

\begin{document}

\title[McKay Correspondence for Hopf actions, II]{McKay
Correspondence for semisimple Hopf actions on regular graded algebras, II}
\author{K. Chan, E. Kirkman, C. Walton and J. J. Zhang}

\address{Chan: Department of Mathematics, Box 354350, University of
Washington, Seattle, Washington 98195, USA}

\email{kenhchan@math.washington.edu}

\address{Kirkman: Department of Mathematics, P. O. Box 7388, Wake Forest
University, Winston-Salem, North Carolina 27109, USA}

\email{kirkman@wfu.edu}

\address{Walton: Department of Mathematics, Temple University, 
Philadelphia, Pennsylvania 19122, USA}

\email{notlaw@temple.edu}

\address{Zhang: Department of Mathematics, Box 354350, University of
Washington, Seattle, Washington 98195, USA}

\email{zhang@math.washington.edu}

\begin{abstract} 
We continue our study of the McKay Correspondence for grading preserving  
actions of semisimple  Hopf algebras $H$ on (noncommutative) Artin-Schelter
regular algebras $A$. Here, we establish correspondences between module 
categories over $A^H$, over $A\#H$, and over $\End_{A^H}(A)$. 
We also study homological properties of (endomorphism rings of) maximal 
Cohen-Macaulay modules over $A^H$.
\end{abstract}

\subjclass[2010]{16E65, 16T05, 16G50, 16D90}

\keywords{Artin-Schelter regular algebra,  Cohen-Macaulay property, 
Gabriel quiver, Hopf algebra action, McKay Correspondence}

\bibliographystyle{abbrv} \maketitle

\setcounter{section}{-1}

\section{Introduction} \label{xxsec0}

Let $\kk$ be an algebraically closed field of characteristic zero. All
algebraic structures in this article are over $\kk$, and we take the
unadorned $\otimes$ to mean $\otimes_{\kk}$. Throughout  
the introduction, let $H$ be a semisimple Hopf algebra, and let $A$ 
be an {\it Artin-Schelter} ({\it AS}) regular algebra  $A$ 
(Definition~\ref{xxdef1.1}) that admits an action of $H$.  By definition
$A$ is a connected graded algebra, and we assume the 
following throughout this paper:

\begin{hypothesis}
Hopf algebra actions on graded algebras are grading preserving.
\end{hypothesis}

The goal of this article is to continue Part I \cite{PartI} of our 
work on the McKay Correspondence, extending  results known for 
$A=\kk [x_1, \dots, x_n]$ and $H = \kk G$ (for $G$ a finite subgroup 
of GL$_n(\kk)$)  to the context of semisimple Hopf actions on AS regular algebras. 
Our setting is motivated as follows. To extend the McKay Correspondence towards noncommutative invariant theory, we use AS regular algebras, a well-behaved homological analogue of commutative polynomial rings. (Note that a commutative AS regular algebra of dimension $n$ generated in degree one must be isomorphic to $\kk [x_1, \dots, x_n]$.) For a good notion of symmetry of a noncommutative (or {\it quantum}) algebra, one must extend beyond the setting of classical symmetry (e.g. group actions); actions of Hopf algebras is a suitable setting of such {\it quantum symmetry}. See \cite{Drinfeld:ICM}, for instance,  for further motivation. In particular, we employ actions of semisimple Hopf algebras $H$ as it is essential for the category of $H$-modules to be semisimple so that the diagrams in this work, including our version of a McKay quiver, can be constructed. Now it is necessary to work with a generalization of both $\kk G$ and $\kk [x_1, \dots, x_n]$ simultaneously as noncommutative algebras do not admit typically many group actions, and moreover, semisimple Hopf actions on commutative domains factor through group actions by \cite{EW}. On the other hand, there is an abundance of symmetries in our chosen setting-- we refer the reader to \cite{CKWZ} for numerous examples of semisimple Hopf actions on (noncommutative) AS regular algebras that do not factor through an action of a group.

\medbreak Using appropriate definitions and hypotheses, which will follow, 
we achieve the results below.  Here, let CM (respectively, MCM) stand for 
Cohen-Macaulay (respectively, maximal Cohen-Macaulay).

\begin{itemize}

\item[ ({\bf A})] 
We establish bijections of indecomposable objects of various left module 
categories of $H$, of the smash product algebra $A\#H$, of the 
fixed subring $A^H$, and of the endomorphism ring $\End_{A^H}(A)$.

\smallskip

\item[({\bf B})] 
We show that the McKay quiver and the Gabriel quiver associated to 
the $H$-action on $A$ are isomorphic as directed graphs.

\smallskip
\item[({\bf C})]  
We produce a version of a theorem of Herzog \cite{He} 
on indecomposable MCM modules over $A^H$, thus adding another 
correspondence to ({\bf A}).

\smallskip

\item[({\bf D}, {\bf E})]  
In \cite{Le1}, Leuschke proved that,
if $R$ is a commutative Cohen-Macaulay ring of dimension $d\geq 2$ 
and of finite Cohen-Macaulay type, then the endomorphism 
ring of the sum of MCM $R$-modules has global dimension $d$.
We present a noncommutative version of results in \cite{Le1} 
on the global dimension (and on the representation dimension) 
of endomorphism rings of MCM modules over $A^H$. 
\end{itemize}

To begin, let us recall a conjecture that we posed in \cite{PartI}: 
an analogue of Auslander's Theorem in the setting of  semisimple 
Hopf actions on AS regular algebras.  Auslander showed that when $G$ 
contains no reflections (e.g. when $G$ is a subgroup of SL$_n(\kk)$) 
and $A = \kk[x_1, \dots, x_n]$, then there is a natural graded algebra 
isomorphism $A\#G \cong \End_{ A^{G}} (A).$

\begin{conjecture} \cite[Conjecture 0.1]{PartI}
\label{xxcon0.2} 
Let $A$ be a noetherian AS regular  algebra that admits an inner 
faithful action of a semisimple Hopf algebra $H$, with trivial 
homological determinant. Then there is a natural graded algebra 
isomorphism $A\#H  \cong \End_{ A^{H}} (A).$
\end{conjecture}

The homological determinant is defined in Definition \ref{xxdef1.5}; 
in the classical setting  (with $H$ a group algebra and $A$ a 
commutative polynomial ring)  the homological determinant is the usual 
determinant of the linear map $g \in G$.
This conjecture was also posed in \cite[Conjecture 0.4]{BHZ1}.
In \cite{PartI}, we provided a partial resolution of this 
conjecture as follows. 

\begin{theorem}\cite[Theorem 0.2]{PartI}
\label{xxthm0.3}
If $A$ is a noetherian AS regular algebra of dimension~$2$ generated in 
degree one, then Conjecture~{\rm{\ref{xxcon0.2}}} is true.
\end{theorem}

There are other cases where Conjecture \ref{xxcon0.2} has been shown to hold, see 
\cite{BHZ2, BHZ1, GKMW}.
Now the precise statement of result ({\bf A}), which depends on 
Conjecture~\ref{xxcon0.2}, is given below.  We call an $A$-module $M$ 
an {\em initial} module if it is a graded module, generated in 
degree 0, with $M_{<0} = 0$.

\begin{introthm}[Proposition~\ref{xxpro2.3}, Corollary~\ref{xxcor2.6}, and Lemma~\ref{xxlem3.17}] 
\label{xxthmA}
{\it Let $A$ be a noetherian AS regular algebra that admits an inner 
faithful action of a semisimple Hopf algebra~$H$,  with trivial 
homological determinant. If Conjecture~{\rm{\ref{xxcon0.2}}} holds, 
then there are bijective correspondences between the isomorphism 
classes of:}
\begin{enumerate}
\item[(a)] 
{\it indecomposable direct summands of $A$ as left 
$A^{H}$-modules;}
\item[(b)] 
{\it indecomposable finitely generated, projective, initial left
$\End_{ A^{H}} (A)$~-modules;}
\item[(c)] 
{\it indecomposable finitely generated, projective, initial left 
$A\# H$-modules; and}
\item[(d)] 
{\it simple left $H$-modules.}
\end{enumerate}
\end{introthm}

In the classical setting, the vertices of the quivers that arise naturally in 
the McKay Correspondence are the objects in 
the theorem above. In this direction, here, we provide an isomorphism 
of  the {\it Gabriel quiver} (Definition~\ref{xxdef2.8}) and the 
{\it McKay quiver} (Definition~\ref{xxdef2.9}) of the 
$H$-action on $A$. 

\begin{introthm}[Theorem~\ref{xxthm2.10}] \label{xxthmB}
{\it The Gabriel quiver ${\mathcal G}(H,A)$ is 
isomorphic to the McKay quiver $(W){\mathcal M}$.}
\end{introthm}

\medbreak The last half of this work (Sections~\ref{xxsec3}-\ref{xxsec5}) 
pertains to CM and MCM modules (Definition \ref{xxdef3.5}(3,4)) over 
the fixed subring $A^H$. Result ({\bf C}) is provided as follows.

\begin{introthm}[Corollary~\ref{xxcor4.5}] 
\label{xxthmC}
{\it Let $A$ be a noetherian AS regular algebra of dimension~$2$ that 
admits an action of  a semisimple Hopf algebra $H$. Then there 
is a bijective correspondence between the modules  of 
Theorem~{\rm{\ref{xxthmA}}}{\rm (a)} and }
\begin{itemize}
\item[(e)]{\it  indecomposable {\rm MCM} left 
$A^H$-modules, up to a degree shift.}
\end{itemize}
\end{introthm}

\medbreak Finally, we draw our attention to endomorphism rings 
of MCM modules over algebras of {\it finite CM type} 
(Definition~\ref{xxdef5.2}), particularly to achieve result ({\bf D}), a noncommutative analogue of \cite[Theorem~6]{Le1}.

\begin{introthm}[Theorem \ref{xxthm5.4}]
\label{xxthmD} 
{\it Let $B$ be a noetherian connected graded algebra with a 
balanced dualizing complex. Suppose that $B$ is {\rm CM} and 
of finite {\rm CM}  type. Let $M$ be a finite direct sum of {\rm MCM} 
right $B$-modules that contain at least one copy of each {\rm MCM} 
module, up to a degree shift. Let $E:=\uEnd_{B}(M)$, the ring of graded endomorphisms of $M$.}
\begin{itemize}
\item[(1)] 
{\it If $E$ is right noetherian, then $E$ has right global
dimension at most $\max\{2,d\}$ where $d=\cd(B)$ 
\textnormal{(}Definition~{\rm{\ref{xxdef3.5}}}\textnormal{)}.}
\item[(2)] 
{\it If $d\geq 2$, then the right global dimension of $E$ is equal to $d$.}
\end{itemize}
\end{introthm}

Note that $E$ in Theorem \ref{xxthmD} (and Theorem \ref{xxthmE} below) 
may not be ${\mathbb N}$-graded, but is always 
${\mathbb Z}$-graded, bounded below, and locally finite. We introduce the notion of 
a {\em commonly graded} module (Definition \ref{xxdef1.2}(4)) to handle this situation. 
When $\cd(B)=2$ and $B$ is module-finite over its affine center, we have a 
stronger result, a noncommutative analogue of \cite[ Proposition~8]{Le1}.

\begin{introthm}[Theorem \ref{xxthm5.7}]
\label{xxthmE} 
{\it Let $B$ be a noetherian connected graded algebra that is {\rm CM} with $\cd(B)=2$, 
of finite CM type, and is module-finite over a central noetherian graded subalgebra $C$. Let $M$ be a finite direct sum of {\rm MCM} right $B$-modules that contains
at least one copy of each {\rm MCM} module, up to a degree shift. 
Then $E:=\uEnd_{B}(M)$ is a noetherian PI generalized AS 
regular algebra \textnormal{(}Definition~{\rm{\ref{xxdef3.9}}}\textnormal{)} 
of dimension $2$.}
\end{introthm}

The paper is organized as follows.  Section 1 reviews required definitions and results 
pertaining to AS regular algebras, Hopf algebra actions, and module categories.  
Section 2 contains  some of the results on the module category correspondences and
 the Gabriel and McKay quivers.  Section 3  provides results on local cohomology and 
the Cohen-Macaulay conditions, including a number of preliminary results extended 
to the commonly graded setting.  Section 4 
contains the generalization of Herzog's Theorem, and  Section 5 contains the 
generalizations of Leuschke's theorems.


\section{Preliminaries} 
\label{xxsec1}

In this section we briefly review background material on 
Artin-Schelter regular algebras, Hopf algebra actions on 
graded algebras,  and some abelian categories
of (certain) modules over a given algebra.

Recall that $\underline{\otimes}$ (respectively, $\uTor, \uHom, \uEnd, \uExt$) 
denotes the grading preserving operation $\otimes$ (respectively, $\Tor, \Hom, \End, \Ext$).

\subsection{Artin-Schelter regular algebras} 
\label{xxsec1.1}

An algebra $A$ is said to be {\it connected graded} if $A =
\kk \oplus A_1 \oplus A_2 \oplus \cdots$ with $A_i \cdot A_j \subseteq
A_{i+j}$ for all $i,j \in \mathbb{N}$.  We consider a class of noncommutative
graded algebras that serve as noncommutative analogues of commutative
polynomial rings. These algebras are defined as follows.

\begin{definition} 
\label{xxdef1.1} 
A connected graded algebra $A$ is called 
\emph{Artin-Schelter {\rm{(}}AS{\rm{)}} Gorenstein} 
if the following conditions hold:
\begin{enumerate} 
\item[(a)]
$A$ has finite injective dimension $d$ on both sides, 
\item[(b)]
$\Ext^i_A(\kk, A) = \Ext^i_{A^{\op}}(\kk, A)= 0$ for
all $i \neq d$ where $\kk = A/A_{\geq 1}$, and 
\item[(c)]
$\Ext^d_A(\kk, A)\cong \kk(\ell)$ and $\Ext^d_{A^{\op}} (\kk, A) 
\cong  \kk(\ell)$ for some integer $\ell$. 
\end{enumerate} 
The integer $\ell$ is called the \emph{AS index} of $A$. If moreover, 
\begin{enumerate} 
\item[(d)] 
$A$ has finite global dimension $d$, 
\item[(e)] 
$A$ has finite Gelfand-Kirillov dimension, 
\end{enumerate} 
then $A$ is called \emph{Artin-Schelter  {\rm{(}}AS{\rm{)}} regular} 
of dimension $d$. 
\end{definition}

The AS regular algebras of global dimension 2 generated in degree 1 
are given in \cite[Example 1.2]{PartI}. Even if $A$ is connected graded,
graded $A$-modules need not be ${\mathbb N}$-graded.
We introduce the following.

\begin{definition}
\label{xxdef1.2} 
Let $M = \bigoplus_{i \in \mathbb{Z}} M_i$ be a ${\mathbb Z}$-graded module (or $\Bbbk$-vector space or 
an algebra).
\begin{enumerate}
\item[(1)]
The {\it lower bound} of $M$ is defined to be
$$b_{l}(M)=\min\{ i\mid M_i\neq 0\};$$
$M$ is called {\it bounded below} if $b_{l}(M)>-\infty$.
\item[(2)]
The {\it upper bound} of $M$ is defined to be
$$b_{u}(M)=\max\{i \mid M_i\neq 0\};$$
$M$ is called {\it bounded above} if $b_{u}(M)<\infty$.
\item[(3)]
$M$ is called {\it locally finite} if $\dim_{\Bbbk} M_i<\infty$
for every $i$.
\item[(4)]
$M$ is called {\it commonly graded}
if it is bounded below and
locally finite. 
\end{enumerate}
\end{definition}

\subsection{Hopf algebra actions} 
\label{xxsec1.2}

Throughout this work $H$ stands for a Hopf algebra over $\kk$ with
 structural notation $(H, m, u, \Delta, \epsilon, S)$. We use
Sweedler notation; namely, $\Delta(h) =: \sum h_1\otimes h_2$. 
The action of $H$ on an algebra $A$ is given as follows.

\begin{definition} 
\label{xxdef1.3}
Let $A$ be an algebra and $H$ a Hopf algebra. 
\begin{enumerate} 
\item[(1)] 
We say that $A$ is a 
{\it \textnormal{(}left\textnormal{)} $H$-module algebra}, 
or that $H$ {\it acts on} $A$, if $A$ is an algebra in the category of left 
$H$-modules. Equivalently, $A$ is a left $H$-module such that
 $h(ab)=\sum h_1(a)
h_2(b)$ and $h(1_A) = \epsilon(h)1_A$
 for all $h\in H$, and all $a, b\in A$. 
\item [(2)]
We say that $A$ is a {\it graded $H$-module algebra} if $A$ is an
algebra in the category of ${\mathbb Z}$-graded (left) $H$-modules 
(with elements in $H$ having degree zero), or equivalently, 
each homogeneous component of $A$ is a left $H$-submodule.
\item[(3)] 
Given an $H$-module algebra, we form 
{\it the smash product algebra $A\# H$}, which is equal to 
$A\otimes H$ as a $\kk$-vector space,  and 
the multiplication in $A\# H$ is given by 
$(a\# h)(a'\# h') = \sum a h_1(a')\# h_2h'$
for all $h, h' \in  H$ and $a, a' \in  A$. 
\end{enumerate} 
\end{definition}

We identify $H$ with a subalgebra of $A\# H$ via the map $i_H :
h\rightarrow 1\#h$ for all $h\in H$, and {we} identify $A$ with a subalgebra
of $A\# H$ via the map $i_A : a \rightarrow a\#1$ for all $a \in A$.
If $A$ is a graded $H$-module algebra, then $A\# H$
is graded with $\deg h=0$ for all $0\neq h\in H$ and $A$ is a graded
subalgebra of $A\# H$. If $A$ is commonly graded 
and $H$ is finite dimensional, then $A\# H$ is commonly graded and 
$(A\# H)_i=A_i\# H$ for all $i$. 

Sometimes it is useful to  to restrict ourselves to Hopf 
($H$-)actions that do not factor
through the action of a proper Hopf quotient of $H$.

\begin{definition} 
\label{xxdef1.4} 
Let $M$ be a left $H$-module. We say that $M$ is an {\it inner faithful} 
$H$-module, or $H$ \emph{acts inner faithfully} on $M$, if $IM\neq 0$ 
for every nonzero Hopf ideal $I$ of $H$. The same terminology applies 
to $H$-module algebras $A$.
\end{definition}

Moreover, the homological determinant of a Hopf algebra action on an
Artin-Schelter algebra is given below. Recall that  a
connected graded algebra $A$ is AS regular if and only if the $\Ext$-algebra 
$E(A):=\bigoplus_{i\geq 0} \Ext^i_A(_A\kk,_A\kk)$ of $A$ is Frobenius \cite[Corollary D]{LPWZ}. 

\begin{definition}
\label{xxdef1.5} 
Retain the notation above. Let $A$ be an AS regular algebra with 
Frobenius $\Ext$-algebra $E$. Suppose ${\mathfrak e}$ is a 
nonzero element in $\Ext^d_A(_A\kk,_A\kk)$, where $d=\gldim(A)$.
Let $H$ be a Hopf algebra acting on $A$ from
the left, and hence $H$ acts on $E$ from the left. 
\begin{enumerate}
\item[(1)] The
{\it homological determinant} of the $H$-module algebra $A$ is 
defined to be $\eta\circ S$, where $\eta: H\to k$ is
determined by 
$h\cdot {\mathfrak e} = \eta(h) {\mathfrak e}$.
\item[(2)]
The
homological determinant is {\it trivial} if $\hdet_H A=\epsilon$.
\end{enumerate}
\end{definition}

\subsection{Module categories} 
\label{xxsec1.3} 
For an algebra $A$, consider the following categories of modules over $A$.

\begin{itemize}
\item  $A\mhyphen \Mod$ (respectively, $\Mod\mhyphen A$):
the category of all left (respectively, right) $A$-modules. 

\smallskip

\item $A\mhyphen \Prm$ (respectively, $\Prm\mhyphen A$): 
the full subcategory of $A\mhyphen \Mod$ (respectively, $\Mod\mhyphen A$) consisting of
projective left (respectively, right) $A$-modules.

\smallskip

\item  $A\mhyphen \Grm$ (respectively, $\Grm\mhyphen A$), when $A$ is graded:
the category of $\mathbb{Z}$-graded  left (respectively, right) $A$-modules.

\smallskip

\item $A\mhyphen \Grprm$ (respectively, $\Grprm\mhyphen A$), when $A$ is graded: 
the full subcategory of  $A\mhyphen \Grm$ (respectively, $\Grm\mhyphen A$) consisting of
projective graded left (respectively, right) $A$-modules.

\smallskip

\item $A\mhyphen \xyz$, when $A\mhyphen \Xyz$ is a category of left  $A$-modules with property
$\Xyz$: the full subcategory of $A\mhyphen \Xyz$ consisting of 
{\it finite} (that is, finitely generated)
modules in $A\mhyphen \Xyz$.

\smallskip

\item $A\mhyphen \Xyz_0$, when $A$ is graded and $A\mhyphen \Xyz$ is a category of left $A$-modules with property $\Xyz$: the full subcategory of 
$A\mhyphen \Xyz$ consisting of {\it initial} modules, that is, 
$M \in A\mhyphen \Xyz$ generated in degree 0 with $M_{<0}=0$.

\smallskip

\item $\add_{A\mhyphen \grm}M$, when $A$ is graded and $M$ is a 
finite graded left $A$-module: the full subcategory of 
$A\mhyphen \grm$ containing all direct 
summands of finite direct sums of degree shifts of $M$.

\smallskip

\item $\xyz \mhyphen A$, $\Xyz_0 \mhyphen A$, and $\add_{\grm \mhyphen A}M$, are defined likewise for right $A$-modules.
\end{itemize}

Unless otherwise stated, we  work with left modules.
However, in the notation $\End_{A^H}(A)$ (or $\End_{B}(M)$), 
$A$ (or $M$) is considered as a right $A^H$-module (or a right
$B$-module), as it was in 
\cite{PartI}. 
Sometimes it is easy to switch from the left to the right 
as the next well-known lemma shows.
A contravariant equivalence between two categories is called a
\emph{duality}. 

\begin{lemma} 
\label{xxlem1.6} 
Let $A$ be an algebra. Then there is a
duality of categories $$\Hom_A(-,A_A): \prm\mhyphen A\to A\mhyphen \prm.$$
As a consequence, there is a bijection between the isomorphism classes of
indecomposable finite projective left $A$-modules and that of
indecomposable finite projective right $A$-modules.
\qed
\end{lemma}

\section{Correspondences between module categories}
\label{xxsec2}
Let $H$ be a semisimple Hopf algebra and $A$ an $H$-module algebra.
In this section we establish several equivalences between categories of
modules over $H$, over the smash product algebra $A \# H$, and  over the
endomorphism ring $\End_{ A^{H}}(A)$.  Some of the
proofs are 
generalizations of results in the commutative setting, and
a good review of the commutative case can be found in \cite{Le2}. We also define the Gabriel
quiver in our setting (Definition~\ref{xxdef2.8}) and establish
that it is isomorphic to the McKay quiver (Definition~\ref{xxdef2.9})  studied in \cite{PartI}.

\subsection{Projective modules over the smash product algebra} 
\label{xxsec2.1}
We begin with a useful lemma that  was proved in \cite{Gu}.

\begin{lemma} 
\label{xxlem2.1} 
Let $H$ be a semisimple Hopf algebra acting on an algebra 
$A$. Let $M$ and $N$ be left $A\# H$-modules. Then the 
following statements hold. 
\begin{enumerate} 
\item[(1)] 
For every $i \geq 0$, $\Ext^i_{A\# H}(M,N)=\Ext^i_A(M,N)^H$. 
\item[(2)] 
$M$ is projective over $A\#H$ if and only if it is 
projective over $A$.
\end{enumerate} 
\end{lemma}

\begin{proof} 
Part (1) is  \cite[Corollary 2.17]{Gu}.

(2) If $M$ is projective over $A\# H$, there is an $A\# H$-module $Q$
such that $M\oplus Q$ is free over $A\# H$. Since $A\# H$ is free over
$A$, then $M\oplus Q$ is free over $A$. Hence, $M$ is projective
over $A$. The converse follows directly from part (1). 
\end{proof}

Next we study the category of graded modules over $A\# H$ when $A$ is a 
{\it connected graded} algebra and $H$ is a semisimple Hopf algebra acting on
$A$. Note that  $A\# H$ is a graded algebra extension of
$A$ with $\deg h=0$ for all $h\in H$.

\begin{lemma} 
\label{xxlem2.2} 
Retaining the notation above, let $M$ be a commonly graded 
left $A\# H$-module. Then the following statements hold. 
\begin{enumerate} 
\item[(1)] 
The projective cover of $M$ is $A\otimes (M/\fm M)$, where  $\fm$ is 
the maximal graded ideal of $A$. 
\item[(2)] 
If $M$ is generated in degree $i$, then the projective cover of $M$ 
is $(A\otimes M_i)[-i]$. 
\end{enumerate} 
\end{lemma}

\begin{proof} 
(1) Since $H$ is semisimple, $M/\fm M$ is isomorphic to a graded 
$H$-submodule of $M$, say $M_0$, that generates $M$ as a left $A$-module. 
Then $A\otimes M_0$ is a graded left $A\# H$-module determined by 
\begin{equation} \label{E2.2.1} \tag{E2.2.1} 
(x\# h)(a\otimes m) = \textstyle \sum xh_1(a)\otimes h_2(m)
\end{equation} 
for all $x,a\in A$, $h\in H$ and $m\in M_0$. Note that $A\otimes M_0$ is a free 
left $A$-module. By Lemma \ref{xxlem2.1}(2), $A\otimes M_0$ is a graded
projective $A\# H$-module. Define a map 
\begin{eqnarray*}
\phi: A\otimes M_0\to M 
&\mathrm{by}&
\phi(a\otimes m)=am
\end{eqnarray*} for all $a\in A$ and $m\in M_0$.
It is easy to check that this map is a surjective $A\# H$-module morphism, and $A \otimes M_0$
 is a projective cover of $M$ as a left $A$-module because
$(A/\fm)\otimes_{A} \phi$ is an isomorphism by the choice of $M_0$.
Therefore $A \otimes (M/\mathfrak{m}M)$ is a projective cover of $M$ as a left $A\# H$-module.

Clearly (2) follows from (1). 
\end{proof}

This lemma prompts the following proposition.

\begin{proposition} 
\label{xxpro2.3} 
Let $H$ be a semisimple Hopf
algebra acting on a connected graded algebra $A$. 
Then we have the following equivalences of categories:
\begin{equation}
\label{E2.3.1}\tag{E2.3.1} 
A\# H\mhyphen \Grprm \simeq  H\mhyphen \Grm,     
\end{equation}
\begin{equation}
\label{E2.3.2}\tag{E2.3.2}
A\# H\mhyphen \Grprm_0 \simeq  H\mhyphen \Mod.  
\end{equation}
\end{proposition}

\begin{proof} We first prove the equivalence  \eqref{E2.3.1}.
Let $\fm$ be the maximal graded ideal of $A$, and we define the functors 
\begin{eqnarray*}
\Phi: A\# H\mhyphen \Grprm \rightarrow H\mhyphen \Grm, && 
P \mapsto P/\fm P=A/\fm \otimes_A P\\
\Psi: H\mhyphen \Grm \rightarrow A\# H\mhyphen \Grprm, && 
M \mapsto A\otimes M. 
\end{eqnarray*}

\noindent
Clearly $\Phi$ maps into $H\mhyphen \Grm$. To show that $\Psi$ maps into $A\# H\mhyphen \Grprm$, 
let $M$ be a graded left $H$-module. Then $A\otimes M$ is a graded 
left $A\# H$-module via \eqref{E2.2.1} (by replacing $M_0$ with $M$). By Lemma 
\ref{xxlem2.1}(2), $A\otimes M_i$ is a graded projective $A\# H$-module 
for each $i$. So $A \otimes M$ is a graded projective $A\# H$-module.

Since $A/\fm\cong \kk$, we have 
$$\Phi \circ \Psi (M)= \Phi( A\otimes M)=\Bbbk\otimes M\cong M,$$ 
which implies that $\Phi\circ \Psi$ is
naturally isomorphic to the identity functor of $H\mhyphen \Grm$. For any
$P$ in $A\# H\mhyphen \Grprm$, $P/\fm P$ generates $P$ by the 
graded Nakayama Lemma. So we can define a surjective map 
$$ \gamma: A\otimes (P/\fm P)   \to P, \quad a\otimes p \mapsto ap.$$
Since $A\otimes (P/\fm P)$ is the projective cover of $P$ by Lemma \ref{xxlem2.2}(1),
and since $P$ is projective, it follows that $\gamma$ is an isomorphism. 
Therefore $\Psi\circ \Phi$ is naturally isomorphic to the identity 
functor of $A\# H\mhyphen \Grprm$.

The equivalence \eqref{E2.3.2} holds by restricting the equivalence  
\eqref{E2.3.1} to the initial graded modules over $A\# H$.
\end{proof}

Let $A$ be a connected graded algebra. Then we have
a minimal free resolution of the left trivial $A$-module $\kk$:
\begin{equation} 
\label{E2.3.3}\tag{E2.3.3} \cdots \to A\otimes
Q_{n}\xrightarrow{\phi_n} A\otimes Q_{n-1}\to \cdots \to A\otimes
Q_1\xrightarrow{\phi_1} A\to \kk\to 0 \end{equation} where each $Q_i$
is a graded $\Bbbk$-vector space and each map $\phi_n$ has entries in $\fm$. 
If $A$ is generated in degree one, then $Q_1$ is in degree one.

\begin{lemma} 
\label{xxlem2.4} 
Let $A$ be a connected graded algebra admitting an action of a semisimple 
Hopf algebra $H$. Let $M$ be a graded bounded below left $A\# H$-module. 
Then there is a minimal projective resolution of the $A\# H$-module $M$ 
that is also a minimal free resolution of the left $A$-module $M$. As a result, 
\eqref{E2.3.3} can be considered as a minimal projective resolution of $\kk$ 
as an $A\# H$-module. 
\end{lemma}

\begin{proof} By Lemmas \ref{xxlem2.1} and \ref{xxlem2.2},
$P_0:=A\otimes (M/\fm M)$ is a projective cover of $M$, which is also a
projective cover of $M$ as a graded $A$-module. Let $M_1=\ker (P_0\to
M)$. Then we have a short exact sequence of graded left $A\# H$-modules
\[0\to M_1\to P_0\to M\to 0.\]
By taking the projective cover of $M_1$,
we obtain $P_1$, which is also the projective cover of $M_1$ as a graded left
$A$-module. The standard process of constructing a minimal projective
resolution leads to the result. 
\end{proof}

\subsection{Modules over endomorphism rings} 
\label{xxsec2.2}
Let $B$ be a graded algebra and let $M$ be a finite graded right $B$-module.
We present below a graded version of a result in \cite{ARS}
that will be used to relate projective modules over the endomorphism ring of
$M$ to $\add_{\grm\mhyphen B} M$; the proof is a straightforward 
adaptation of that in \cite{ARS}.

\begin{proposition} \cite[Proposition~II.2.1]{ARS}
\label{xxpro2.5} 
Let $B$ be a commonly graded algebra and let $M$ be a finite graded 
right $B$-module.  Denote by $E$ the ring 
$\underline{\End}_{\grm\mhyphen B}(M)$ of graded endomorphisms 
of $M$. Then the evaluation functor
\begin{eqnarray*}
e_M:=\uHom_{\grm\mhyphen B} ( M,- ) : 
~\grm\mhyphen B & \longrightarrow & \grm\mhyphen E
\end{eqnarray*}
has the following properties.
\begin{enumerate}
\item[(1)] 
We have that 
$$\uHom_{\grm\mhyphen B} (Y,Z) ~~ \cong ~~\uHom_{\grm\mhyphen E} (e_M(Y), e_M(Z)) $$ 
for $Y \in \add_{\grm\mhyphen B}M$ and $Z \in \grm\mhyphen B$.
\item[(2)] 
If $X$ is in $\add_{\grm\mhyphen  B} M$, then $e_M(X)$ is in $\grprm\mhyphen E$.
\item[(3)] 
$e_{M}|_{\add_{\grm\mhyphen  B} M }$ induces an equivalence of categories 
$$\add_{\grm\mhyphen  B} M \simeq \grprm\mhyphen E.$$

\vspace{-.25in}

 \qed
\end{enumerate}
\end{proposition}

Now the results above yield, in our context of semisimple Hopf actions on 
AS regular algebras,  part of the McKay Correspondence pertaining to 
module categories over $H$, over $A\# H$, and over 
$\End_{A^{H}}(A)$. 

\begin{corollary} 
\label{xxcor2.6} 
Let $A$ be a noetherian AS regular algebra and $H$ be a semisimple Hopf 
algebra that acts on $A$ inner faithfully with trivial homological 
determinant so that  Conjecture {\rm{\ref{xxcon0.2}}} 
holds. Then there are natural bijections between isomorphism classes of 
\begin{enumerate}
\item[(a)]
indecomposable objects in $A\#H \mhyphen \grprm$,
\item[(b)]
indecomposable objects in $\uEnd_{A^{H}} (A) \mhyphen
\grprm$,
\item[(c)]
indecomposable right $A^{H}$-modules in $\add_{\grm\mhyphen A^H} A$.
\end{enumerate}
\end{corollary}

\begin{proof} 
Conjecture~\ref{xxcon0.2} gives a graded algebra isomorphism $A\#H 
\cong \uEnd_{A^{H}} (A)$, 
which induces a bijection between (a) and (b). Proposition \ref{xxpro2.5}(3) 
for $M=A$ and $ B = A^H$ induces an equivalence 
\[
\add_{\grm\mhyphen A^H}A~\simeq~\grprm\mhyphen\uEnd_{ A^{H}} (A).
\]
In particular, this equivalence gives a bijection between the indecomposable 
objects in the respective categories. Combining these facts with Lemma~\ref{xxlem1.6}  
gives a bijection between (b) and (c).
\end{proof}

\subsection{The Gabriel and McKay quivers}  
\label{xxsec2.3} 
Let us consider the following notation that will be used 
throughout this subsection.

\medskip

\noindent {\it Notation.} 
Let $H$ be a semisimple Hopf algebra that acts on a connected graded 
algebra $A$ so that $A$ is finitely generated in degree one by a 
left $H$-module $W$. Moreover, let  $\{V^{(0)}, V^{(1)}, \dots, V^{(s)}\}$ 
be a complete set of nonisomorphic simple left $H$-modules, where 
$V^{(0)}=H/(\ker \epsilon)$, and let $P^{(j)}$ be the left $A\# H$-module 
$A \otimes V^{(j)}$ as defined by \eqref{E2.2.1}.

\medskip

The result below holds by Proposition~\ref{xxpro2.3} and Lemma \ref{xxlem2.4}, 
using $H$ and the maximal graded ideal of $A$ in the roles of $\kk G$ 
and $\mathfrak m$, respectively, in  \cite[Section~1]{Au2}. See, also, 
\cite[Corollary~5.19]{LW}. 

\begin{lemma} 
\label{xxlem2.7}
Let $j$ be an integer between $0$ and $s$.
\begin{enumerate}
\item[(1)] 
The complete set of nonisomorphic, indecomposable, initial, projective 
left $A\# H$-modules is $\{P^{(0)}, P^{(1)}, \dots, P^{(s)}\}$. 
\item[(2)] 
The module $V^{(j)}$ is a simple left $A \#H$-module via the surjection \linebreak
$A \# H \to H\to V^{(j)}$, with projective cover $P^{(j)}$.
\item[(3)]
Suppose 
\begin{equation}
\label{E2.7.1}\tag{E2.7.1}
\cdots \to Q_n^{(j)}\to  \cdots \to Q_1^{(j)} \to Q_0^{(j)} 
\rightarrow V^{(j)} \rightarrow 0
\end{equation}
is a minimal projective resolution of the $A\#H $-module $V^{(j)}$, then the minimal 
projective resolution of $V^{(j)}$ over $A$ also has this form. 
\end{enumerate}
\end{lemma}

\begin{proof} (1,2) Clear.

(3) This follows from Lemma \ref{xxlem2.4}.
\end{proof}

Now we define the quivers of interest.

\begin{definition} 
\label{xxdef2.8} 
The \emph{Gabriel quiver} ${\mathcal G}(H,A)$ of the $H$-action on $A$ is the directed graph with
vertices  $P^{(0)},P^{(1)},\dots, P^{(s)}$, and with
$m_{ij}$ arrows from $P^{(i)}$ to $P^{(j)}$, where $m_{ij}$ is the 
multiplicity of $P^{(i)}[-1]$ in $Q_1^{(j)}$ as defined in \eqref{E2.7.1}.
\end{definition}

\begin{definition} \cite[Definition 2.3]{PartI} 
\label{xxdef2.9}
Let $W$ be a finite dimensional left $H$-module. The 
{\it McKay quiver} $(W)\mathcal{M}$ has 
vertices $V^{(0)},V^{(1)},\dots, V^{(s)}$, and $n_{ij}$ arrows from 
$V^{(i)}$ to $V^{(j)}$ where 
$W \otimes V^{(j)} = \bigoplus_{i=0}^s (V^{(i)})^{\oplus n_{ij}}$. 
The set of values $n_{ij}$ are referred to as the {\it fusion rule coefficients}.
\end{definition}

The following result, our Theorem B, was proved by Auslander  \cite[Section~1]{Au2} 
when $A$ is a commutative polynomial ring or a formal power series 
ring and $H$ is a group algebra.

\begin{theorem} 
\label{xxthm2.10} 
The Gabriel quiver ${\mathcal G}(H,A)$ is 
isomorphic to the McKay quiver $(W){\mathcal M}$.
\end{theorem}

\begin{proof}
By Proposition~\ref{xxpro2.3}, the number of vertices of 
${\mathcal G}(H,A)$ and of $(W)\mathcal{M}$ are the same; namely, 
$P^{(j)}$ corresponds to $V^{(j)}$ for $j=0, \dots, s$.
So it suffices to show that the values $m_{ij}$ in 
Definition~\ref{xxdef2.8} form the fusion rule coefficients of Definition~\ref{xxdef2.9}.

Write a minimal projective resolution of the 
trivial $A$-module $V^{(0)}=\kk$ as 
$$
\cdots \to A\otimes W_n\to \cdots \to A\otimes W_{2}
\to A\otimes W[-1]\to A\to \kk\to 0,
$$
which is also a minimal projective resolution of $\kk$ as a graded left
$A\#H$-module by Lemma~\ref{xxlem2.7}(3). 
Tensoring with $V^{(j)}$ from the right
yields a minimal projective resolution of the left $A\# H$-module 
$V^{(j)}$, 
$$\cdots \to 
\cdots \to (A\otimes W_{2})\otimes V^{(j)} \to (A\otimes W[-1])\otimes V^{(j)}
\to A\otimes V^{(j)}\to V^{(j)}\to 0.$$ 
Note that $A\otimes V^{(j)}=P^{(j)}$ and
$(A\otimes W[-1])\otimes V^{(j)}=Q_1^{(j)}$ as in Lemma~\ref{xxlem2.7}. We obtain,
by Definitions \ref{xxdef2.8} and \ref{xxdef2.9}, 
$$\begin{aligned}
{\textstyle \bigoplus_{i=0}^s } \left( P^{(i)}[-1]\right)^{\oplus m_{ij}}
&\cong Q_1^{(j)} \\
&=(A\otimes W[-1])\otimes V^{(j)} \cong A\otimes (W[-1]\otimes V^{(j)}) \\
&\cong A\otimes \left(\textstyle \bigoplus_{i=0}^s (V^{(i)})^{\oplus n_{ij}}\right)[-1]\\
&\cong {\textstyle \bigoplus_{i=0}^s } \left( P^{(i)}[-1]\right)^{\oplus n_{ij}}.
\end{aligned}
$$
Therefore, $m_{ij}=n_{ij}$ for all $i,j$, as desired. 
\end{proof}

\section{Local cohomology and the Cohen-Macaulay property}
\label{xxsec3}

In this section we recall some basic definitions related to local
cohomology and Cohen-Macaulay modules needed for the results in 
Sections~\ref{xxsec4} and~Section~\ref{xxsec5}. We refer
to \cite{AZ,Jo1, VdB, Ye} for details and some undefined
terminology.  We will also extend a number of results from the connected graded  to the commonly graded setting.

\begin{definition} \label{xxdef3.1} 
Let $A$ be a commonly graded algebra 
(Definition \ref{xxdef1.2}(4)), which is ${\mathbb Z}$-graded, but
may not be ${\mathbb N}$-graded. 
\begin{enumerate}
\item[(1)] \cite[Definition 1.7.3]{NV} The {\it graded Jacobson
radical}  of $A$ is
$$\fm_A:=\bigcap\{ L \mid {\text{$L$ maximal  graded left ideals of $A$}}\}.$$
\item[(2)] $A$ is called {\it graded semilocal} if $A/\mathfrak{m}_A$ is graded semisimple (\cite[Section~1.7]{NV}).
\item[(3)] For a graded left $A$-module $M$, the 
{\it graded $\Bbbk$-linear dual of $M$} is 
$$M^\ast:=\bigoplus_{i\in {\mathbb Z}} \Hom_{\Bbbk}(M_{-i}, \Bbbk).$$
It is easy to see that $M^\ast$ is a graded right $A$-module. 
\end{enumerate}
\end{definition}

We start with an easy lemma. 

\begin{lemma}
\label{xxlem3.2} 
Let $A$ be a commonly graded algebra. 
\begin{enumerate}
\item[(1)] 
If $L$ is a maximal graded left ideal of $A$, then 
$L\supseteq \bigoplus_{i> -b_{l}(A)} A_i$. As a 
consequence, $\fm_A \supseteq \bigoplus_{i> -b_{l}(A)} A_i$.
\item[(2)]
If $S=A/L$, where $L$ is a maximal graded left ideal of $A$,
then $b_{l}(S)\geq b_{l}(A)$ and $b_{u}(S)\leq -b_{l}(A)$.
As a consequence, $S$ is finite dimensional.
\item[(3)]
There are only finitely many isomorphism classes of simple
graded left $A$-modules, up to degree shifts.
As a consequence, $A$ is graded semilocal.
\item[(4)] \cite[Lemma 1.7.4(4)]{NV}
$$\fm_A=\bigcap \{l.ann(S)\mid {\text{$S$ graded simple left $A$-modules}}\}.$$
\item[(5)]  \cite[Lemma 1.7.4(7)]{NV}
$$\begin{aligned}
\fm_A&=\bigcap\{ R \mid {\text{$R$ maximal graded right ideals of $A$}}\}\\
&=\bigcap \{l.ann(S)\mid {\text{$S$ graded simple right $A$-modules}}\}.
\end{aligned}
$$
\item[(6)]
If $M$ is a graded left $A$-module, then there is a natural 
isomorphism
$$M^\ast \cong \uHom_A(M,A^\ast).$$
As a consequence, $A^\ast$ is a graded injective left 
{\rm{(}}or right{\rm{)}} $A$-module. 
\end{enumerate}
\end{lemma}

\begin{proof} (1) Let $x\in \bigoplus_{i> -b_{l}(A)} A_i$. 
Then $ b_{l}(Ax)\geq \deg x+ b_{l}(A)>0$ and $(Ax+L)_0=L_0$. Then 
$Ax+L\neq A$. Since $L$ is maximal, $Ax+L=L$ or $Ax\subseteq L$.
Thus $x\in L$ and consequently, $\bigoplus_{i> -b_{l}(A)} A_i
\subseteq L$.

(2) This follows from part (1) and the fact that $S=A/L$
for some maximal graded left ideal $L$.

(3) 
Let $S$ be a graded simple left $A$-module. Then 
$\fm_A S=0$. Thus $S$ is a factor module of $A/\fm_A$
up to a shift. Since $A/\fm_A$ is finite
dimensional by part (1), there are only finitely
many isomorphism classes of graded simple left
$A$-modules, up to a shift.

(4,5) These are properties of the graded Jacobson radical.

(6) For every graded left $A$-module $M$, one has 
\[\begin{array}{rll}
M^\ast &=\uHom_{\Bbbk}(M,\Bbbk) &\cong
\uHom_{\Bbbk}(A\otimes_A M, \Bbbk)\\
&\cong \uHom_{A}(M,\uHom_{\Bbbk}(A,\Bbbk))
&=\uHom_A(M, A^\ast);
\end{array}\]
see \cite[Section~1.2]{NV}.
Since the functor $(-)^\ast$ is exact, $A^\ast$
is an injective graded left $A$-module.
\end{proof}

Part (2) of the following is a graded version of
Nakayama's lemma (see, also, \cite[Lemma 1.7.5]{NV}).

\begin{proposition}
\label{xxpro3.3} 
Let $A$ be a commonly graded algebra. 
\begin{enumerate}
\item[(1)] 
Suppose $A$ is left noetherian.
For each $i>0$, there are $j$ and $k$ such that
$$A_{\geq k}\subseteq (\fm_A)^j \subseteq A_{\geq i}.$$
As a consequence, the two sequences 
$\{\fm_A^n\}_{n}$ and $\{A_{\geq n}\}_n$
are cofinal. 
\item[(2)]
Let $M$ be a bounded below left $A$-module.
If $\fm_A M=M$, then $M=0$.
\end{enumerate}
\end{proposition}

\begin{proof}
(1) Let $I$ be the ideal $A (A_{\geq 1-2b_{l}(A)})A$ 
which is a subspace of $A_{\geq 1}$. Let $B=A/I$. Then 
$B$ is a finite dimensional graded algebra, and 
$(\fm_B)^N=0$ for some integer $N>0$. This means
that $(\fm_A)^N\subseteq I\subseteq A_{\geq 1}$. 
Therefore $(\fm_A)^{iN}\subseteq A_{\geq i}$ for all
$i$, and so we can take $j=iN$.

Now we assume that $A$ is left noetherian. Then 
for each $j$, $A/(\fm_A^j)$ is finite dimensional.
This means that $A_{\geq k}\subseteq \fm_A^j$ 
for some $k$. 

(2) By the proof of part (1), even without the left noetherian
hypothesis, one has $\fm_A^j\subseteq A_{\geq i}$.

Suppose that $M\neq 0$ and that $\fm_A M=M$.
Then $\fm_A^n M=M$ for all $n\gg 0$. By part (1), 
we can choose $n$ such that $\fm_A^n\subseteq A_{\geq 1}$.
Then 
$$b_{l}(M)=b_{l}(\fm_A^n M)\geq b_{l}(\fm_A^n) +b_{l}(M)
\geq 1+ b_{l}(M)>b_{l}(M),$$ 
yielding a contradiction. Therefore $M=0$. 
\end{proof}

Let $A$ be a left noetherian commonly graded algebra and 
$\fm_A$ (or $\fm$ if no confusion occurs) be the graded
Jacobson radical of $A$. For each
graded left $A$-module $M$, the $\fm_A$-torsion submodule of $M$ is
defined to be 
\smallskip
$$\Gamma_{\fm_A}(M):=\{x\in M\mid A_{\geq n}x=0 \quad
{\text{for some $n\geq 1$}}\}=\lim_{n\to \infty} \uHom_A(A/(\fm_A)^n,M).$$ 

\smallskip

\noindent The functor $\Gamma_{\fm_A}$ is a left
exact functor from $A\mhyphen \Grm$ to itself. 

We now extend some definitions to the commonly graded case.

\begin{definition}
\label{xxdef3.4}
The $i$-th right derived functor $R^i\Gamma_{\fm_A}(-)$ is called 
the {\it $i$-th local cohomology}, and $R^i\Gamma_{\fm_A}(M)$ is 
called the {\it $i$-th local cohomology module of $M$}. 
\end{definition}

\begin{definition} 
\label{xxdef3.5} 
Let $A$ be a noetherian commonly graded 
algebra, and let $M$ be a finite left $A$-module. 
\begin{enumerate} 
\item[(1)] 
\cite{AZ}
The {\it cohomological dimension of $M$} and {\it of $A$} are given as, respectively,  
\begin{eqnarray*}
\cd(A)&:=&\max\{i\mid R^i\Gamma_{\fm_A}(N)\neq 0 \quad 
{\text{for some $N\in A\mhyphen \Grm$}}\}\\
\cd(M)&:=&\max\{i\mid R^i\Gamma_{\fm_A}(M)\neq 0\}.
\end{eqnarray*}
\item[(2)] 
\cite[Proposition 4.3]{Jo1}
The {\it depth of $M$} is defined to be 
\begin{eqnarray*}
\dep(M)&:=&\min\{i\mid R^i\Gamma_{\fm_A}(M)\neq 0\}.
\end{eqnarray*}
\item[(3)] 
$M$ is called \emph{Cohen-Macaulay} (CM, or $d$-CM) 
if $d:=\cd(M) < \infty$ and $\dep(M)=\cd(M)$.

\smallskip

\item[(4)] 
 $M$ is  called \emph{Maximal Cohen-Macaulay} 
(MCM) if $\cd(A)<\infty$ and $M$ is $\cd(A)$-CM. 

\smallskip

\item[(5)] 
$A$ is  called \emph{Cohen-Macaulay} if it 
is MCM as a left and right $A$-module.
\end{enumerate} 
\end{definition}

Note that when $A$ is commutative, noetherian, and connected
graded (or local), the definition of CM above is equivalent to the
standard definition of CM using the $\Ext$-group or using the maximal
length of regular sequences \cite[Chapter 2]{BH}. 

 \medbreak We recall a few more definitions.

\begin{definition}
\cite[Definition 3.2]{AZ} 
\label{xxdef3.6} 
Let $A$ be a noetherian commonly graded 
algebra. We say $A$ satisfies the \emph{left $\chi$-condition}, 
if $\uExt^i_A(S, M)$ is finite dimensional for every simple left
graded $A$-module $S$, every noetherian graded left $A$-module 
$M$, and every $i\geq 0$. 
The \emph{right $\chi$-condition} is defined similarly. If 
both left and right $\chi$ hold, then we say $A$ satisfies $\chi$. 
\end{definition}

It is not difficult to check that this definition is equivalent to the
original definition of $\chi$ given in \cite[Definition 3.2]{AZ}.

The following definition is due to Yekutieli \cite{Ye} and Van
den Bergh \cite{VdB}. Let $\D^b(A\mhyphen \Mod)$ denote the bounded
derived category of left $A$-modules. Let $A^{\op}$ denote the opposite
ring of $A$, and let $A^{\e}$ denote the enveloping algebra $A\otimes A^{\op}$.

\begin{definition} 
\label{xxdef3.7} 
Let $A$ be a noetherian algebra. 
\begin{enumerate} 
\item[(1)] \cite{Ye} 
A complex $D\in \D^b(A^{\e}\mhyphen \Mod)$ is called a 
\emph{dualizing complex over $A$} if it satisfies the following 
conditions: 
\begin{enumerate} 
\item[(i)] 
$D$ has finite injective dimension over $A$ and over $A^{\op}$;
\item[(ii)] 
for every $i$, the $i$-th cohomology $H^i(D)$ is finite over $A$
and over $A^{\op}$, respectively; 
\item[(iii)] 
the canonical maps $A \to \RHom_{A}(D,D)$ and $A \to 
\RHom_{A^{\op}}(D,D)$ are isomorphisms in $\D(A^{\e}\mhyphen \Mod)$. 
\end{enumerate} 
\item[(2)] 
\cite{Ye}
Assume that $A$ is commonly graded. 
Then a graded dualizing complex $D\in \D^b(A^{\e}\mhyphen \Grm)$ is 
called \emph{balanced} if
$$ R\Gamma_{\fm_A}(D)^*~\cong ~R\Gamma_{\fm_{A^{\op}}}(D)^*~\cong ~A$$ 
as $A$-bimodules. 
\item[(3)]
\cite{VdB} 
A dualizing complex $D$ over
$A$ is called \emph{rigid} if there is an isomorphism 
$$
D ~\cong ~ \RHom_{A^{\e}}(A,D\otimes D^{\op})
$$
in $\D(A^{\e}\mhyphen \Mod)$. Here
the left $A^{\e}$-module structure of $D\otimes D^{\op}$ comes from the
left $A$-module structure of $D$ and the left $A^{\op}$-module structure
of $D^{\op}$.
\end{enumerate} 
\end{definition}

Now we have the following existence theorems;  the original versions 
of some of these results were given for connected graded algebras
\cite{VdB}. Some of these results were extended to the semilocal case in
\cite{WZ1}. Since commonly graded algebras are graded semilocal by Lemma~\ref{xxlem3.2}(3),
one can verify the  result below following ideas in \cite{VdB, WZ1}.

\begin{theorem} 
\label{xxthm3.8} 
Let $A$ be a noetherian commonly graded algebra. 
\begin{enumerate} 
\item[(1)] 
\cite[Corollary~4.14]{Ye} 
If $A$ is AS regular, then $A$ has a balanced dualizing 
complex.
\item[(2)] 
\cite[Theorem 6.3]{VdB} 
In general, $A$ has a balanced dualizing complex if and only if
$\cd(A)$ and $\cd(A^{\op})$ are finite, and 
$A$  and $A^{\op}$ satisfy $\chi$. 
\item[(3)] 
\cite[Theorem 6.3]{VdB} 
If $A$ has a balanced dualizing complex $D$, then 
$D\cong R\Gamma_{\fm}(A)^*$. 
\item[(4)]
\cite[Proposition 8.2(2)]{VdB} 
If $A$ has a balanced dualizing complex $D$, then it is a 
rigid dualizing complex.
\item[(5)]
\cite[Proposition 8.2(1)]{VdB} 
If they exist, then the rigid dualizing complex and 
the balanced dualizing complex over $A$ are unique up to 
isomorphism.
\item[(6)] 
\cite[Theorem 5.1]{VdB} 
Let $D$ be the balanced dualizing complex over $A$, if it exists. Then 
$$
R\Gamma_{\fm}(M)^{\ast}~=~\RHom_A(M,D)
$$
for any $M\in\D(A\mhyphen \Grm)$. 
\item[(7)]
\cite[Theorem 4.2(3,4)]{YZ1} 
Let $D$ be the balanced dualizing complex over $A$, if it exists. Then
$$
\quad \quad \cd(A)=\cd(A^{\op})=-\min\{i\mid H^i(D)\neq 0\} =\max\{i \mid
R^i\Gamma_{\fm}(A)\neq 0\}.
$$ 
\end{enumerate} 
\vspace{-.22in}

\qed
\end{theorem}

Next we introduce another analogue of the Gorenstein and regularity
properties (similar to \cite[Definition 3.3]{RRZ}) by employing the 
dualizing complexes above.

\begin{definition} 
\label{xxdef3.9}
Let $A$ be a noetherian commonly graded algebra.
\begin{enumerate}
\item[(1)] 
We say $A$ is {\em{generalized AS Gorenstein}} if
\begin{enumerate}
\item[(i)]
$A$ has finite injective dimension $d$,
\item[(ii)]
$A$ has a balanced dualizing complex, and
\item[(iii)] 
$R^{i} \Gamma_{\fm} (A)^{\ast} \cong \begin{cases} \Omega & i=d\\
0& i\neq d \end{cases}$, \; where $\Omega$ is an invertible graded
$A$-bimodule. 
\end{enumerate}
\item[(2)] 
We say $A$ is {\em{generalized AS regular}} if $A$ is
generalized AS Gorenstein of finite global dimension $d$.
\end{enumerate}
\end{definition}

\begin{remark}
\label{xxrem3.10}
\begin{enumerate}
\item[(1)] The original definition of AS regularity 
(Definition~\ref{xxdef1.1}) requires that
$A$ is connected graded. 
There is a condition missing in the definition 
given in \cite[Definition 3.3]{RRZ}, which is ``$R^i\Gamma_{\fm}(A)=0$ 
for all $i\neq d$''.  
In \cite[Definition 3.3]{RRZ}, one considers 
${\mathbb N}$-graded (but not necessarily connected) algebras.
Here we are considering algebras that are not necessarily ${\mathbb N}$-graded.
\smallbreak 

\item[(2)] If $A$ is  generalized AS Gorenstein, 
by Theorem \ref{xxthm3.8}(2-6) the rigid
dualizing complex over $A$ is $\Omega [d]$.
\smallbreak 

\item[(3)] In general, $\Omega$ is not isomorphic to $A(-\ell)$ as a graded left 
$A$-modules. For example, let $A=B\oplus C$, where $B$ and $C$ are AS regular 
in the sense of Definition \ref{xxdef1.1} with the same global
dimension but different AS indices. Then $\Omega=B(-\ell_B)\oplus C(-\ell_C)$,
which cannot be isomorphic to a shift of $A$ as a graded left $A$-module.
\end{enumerate}
\end{remark}

We end this section by providing preliminary lemmas about depth, 
CM modules and  MCM modules that are needed in the rest of the paper. 
We also  compare the notion of regularity above with AS
regularity as in Definition~\ref{xxdef1.1}. Note that the global dimension of 
the smash product algebra $A\# H$ is equal to the global dimension of 
$A$ when $H$ is semisimple.

\begin{lemma} 
\label{xxlem3.11}
Let $A$ be a noetherian  commonly graded algebra 
and let $M$ be a finite left $A$-module. Let $B$ be a graded subalgebra of 
$A$ such that $_{B} A$ and $A_{B}$ are finite  $B$-modules. Take 
${}_B M$ to be the left $B$-module corresponding to $M$.
\begin{enumerate}
\item[(1)] {\cite[Theorem 8.3]{AZ}} 
We have that $R^{i} \Gamma_{\fm_{B}} ({}_B M) =R^{i} \Gamma_{\fm_{A}} (M)$. 
As a consequence,  $\cd (A) \leq \cd (B)$.
\item[(2)] 
The module $_{B} M$ is {\rm CM} if and only if $_{A} M$
is. If $\cd (A) =\cd (B) < \infty$, then $_{B} M$ is {\rm MCM} if 
and only if $_{A} M$ is {\rm MCM}.
\item[(3)] 
Suppose $M=M_{1} \oplus M_{2}$ where $M_{i} \neq 0$. Then $M$ is {\rm CM} 
\textnormal{(}respectively, {\rm MCM}\textnormal{)} if and only if both $M_{1}$ and $M_{2}$ are 
{\rm CM} \textnormal{(}respectively, {\rm MCM}\textnormal{)} of the same depth. 
\item[(4)] 
If $A$ is connected graded AS regular, then $\cd (A) =
\injdim(A)$ and $A$ is {\rm MCM}.
\item[(5)] If $A$ is generalized AS regular, then $\cd (A) =
\injdim(A)$ and $A$ is {\rm MCM}.
\end{enumerate}
\end{lemma}

\begin{proof} (2) This part follows easily from part (1).

(3) This follows from the fact that the functor $R^i\Gamma_{\fm_A}(-)$
is additive.

(4) This follows from the definitions and an easy computation.

(5) By Definition~\ref{xxdef3.9}(1), $R^i\Gamma_{\fm_A}(A)=0$ for 
all $i\neq \injdim(A)$ and $R^i\Gamma_{\fm_A}(A)\neq 0$ when 
$i=\injdim(A)$. The assertion follows from the fact that 
$\cd(A)=\injdim(A)$, due to Theorem \ref{xxthm3.8}(7).
\end{proof}

\begin{proposition} 
\label{xxpro3.12} 
Let $A$ be noetherian commonly graded.
\begin{enumerate} 
\item[(1)]  
If $A$ is AS  regular, then $A$ is generalized AS  regular.
\item[(2)] \cite[Remark~3.6]{RRZ}
If $A$ is connected graded generalized AS regular, then $A$ is
AS regular.
\item[(3)]
\cite[Theorem 4.1(b)]{RRZ} 
Let $H$ be a semisimple Hopf algebra acting on a generalized AS 
regular algebra $A$. Then $A\#H$ is generalized AS regular.
\end{enumerate}
\end{proposition}

\begin{proof}
(1) Each noetherian connected graded AS regular algebra has a 
balanced dualizing complex of the form ${^\mu A^1}(-\ell)[d]$ 
by \cite[Corollary 4.14]{Ye}, and is 
MCM with  $\cd(A) = \cd(A^{\op}) = \injdim(A)$  by 
Theorem~\ref{xxthm3.8}(7) and Lemma~\ref{xxlem3.11}(4). 
Hence $A$ satisfies Definition~\ref{xxdef3.9}.

(2) \cite[Remark~3.6]{RRZ} is still valid. 

(3) The proof is the same as \cite[Theorem 4.1(b)]{RRZ}.
\end{proof}

\begin{lemma} 
\label{xxlem3.13} 
Suppose $A$ is a  generalized AS regular algebra and $M$ 
is a finite graded left $A$-module. Then $M$ is {\rm MCM} 
if and only if $M$ is projective. 
\end{lemma}

\begin{proof} 
By definition, $A$ is MCM. By Lemma~\ref{xxlem3.11}(3), 
every finite projective $A$-module is MCM.  For the converse, 
let $M$ be a MCM $A$-module. It is easy to check that $A$ is 
depth-homogeneous in the sense of \cite[page~521]{WZ3}. Now by the
noncommutative version of the Auslander-Buchsbaum formula for 
(not necessarily connected)
graded rings \cite[Theorem 3.4]{WZ3}, $M$ is projective. 
\end{proof}

\begin{lemma}
\cite[Corollary~4.8]{VdB}
\label{xxlem3.14} Let $A$ and $B$ be  noetherian commonly graded 
algebras with balanced dualizing complexes. Suppose $M$ is an 
$(A,B)$-bimodule that is a finite module over $A$ and over $B$. Then
$R^i\Gamma_{\fm_A}(M)=R^i\Gamma_{\fm_{B^{\op}}}(M)$, and,  as a consequence,
$\dep(_AM)=\dep(M_B)$. \qed
\end{lemma}

\begin{lemma} 
\label{xxlem3.15} 
Let $M$ and $N$ be nonzero finite left
$A$-modules related by the exact sequence 
$$0\to M\to P_{s-1}\to P_{s-2}\to \cdots \to P_0\to N\to 0.$$ 
Then $\dep(M)\geq \min\{\dep(N)+1,~\dep(P_0), \dots, 
~\dep(P_{s-2}), ~\dep(P_{s-1})\}$. If, further,
\linebreak $\dep(P_j)\geq s+\dep(N)$ for each $j$, then
$\dep(M)=\dep(N)+s$. 
\end{lemma}

\begin{proof} By induction on $s$, it suffices to show the result in the case 
$s=1$. Letting $P=P_0$ and applying $R\Gamma_{\fm}(-)$ to the short exact
sequence $0\to M\to P\to N\to 0$, we obtain 
\[\cdots \to
R^{i-1}\Gamma_{\fm}(N)\to  R^{i}\Gamma_{\fm}(M)\to
R^{i}\Gamma_{\fm}(P)\to  R^{i}\Gamma_{\fm}(N)\to
R^{i+1}\Gamma_{\fm}(M)\to \cdots.
\] Since $R^i\Gamma_{\fm}(P)=0$ for all
$i<\dep(P)$, we get $R^{i}\Gamma_{\fm}(M)\cong  R^{i-1}\Gamma_{\fm}(N)$. 
The latter is equal to 0 for
all $i\leq \dep(N)$. So, for $i=\dep(N)+1$, one has 
\begin{eqnarray*}
0\neq R^{i-1}\Gamma_{\fm}(N)&\subseteq&  R^{i}\Gamma_{\fm}(M),
\end{eqnarray*}
which implies
that $\dep(M)=\dep(N)+1$, as required. 
\end{proof}

\begin{lemma} 
\label{xxlem3.16} 
Let $A$ and $B$ be noetherian commonly graded algebras equipped 
with balanced dualizing complexes. Suppose that $M$ is a finite 
right $B$-module and $N$ is an $(A,B)$-bimodule that is a finite 
module over $A$ and over $B$. Then 
\begin{eqnarray*}
\dep(\Hom_{B^{\op}}(M,N))&\geq&
\min\{2,\dep(N)\}.
\end{eqnarray*}
\end{lemma}

\begin{proof} 
Consider a free resolution of the right $B$-module $M$:
$$\cdots \to P_1\to P_0\to M\to 0,$$ 
where $P_i$ is finite for $i=0,1$.
By applying $\Hom_{B^{\op}}(-,N)$ to the exact sequence above, one
has the sequences
$$0\to \Hom_{B^{\op}}(M, N)\to \Hom_{B^{\op}}(P_0, N)\to\coker_1\; \to 0,$$ 
and
$$0\to \coker_1\;\to \Hom_{B^{\op}}(P_1, N)\to \coker_2\; \to 0.$$ 
Since $P_i$ is projective over $B$, $\Hom_{B^{\op}}(P_i, N)$ has (left) 
depth at least equal to $\dep(_AN)$, which also equals $\dep(N_B)$ by 
Lemma \ref{xxlem3.14}. Without loss of generality  assume that
$\dep(N)\geq 1$ and $\dep(\Hom_{B^{\op}}(P_0, N))\geq 1$. If
$\Hom_{B^{\op}}(P_0, N)$ has depth~1, then, by Lemma \ref{xxlem3.15}
$\dep(\Hom_{B^{\op}}(M, N))\geq 1$, as claimed. Now if $\dep(N)\geq 2$, then
$\dep(\Hom_{B^{\op}}(P_i, N)) \geq 2$. It follows from  the short
exact sequences above and Lemma \ref{xxlem3.15} that
$\dep(\coker_1)\geq 1$ and $\dep(\Hom_{B^{\op}}(M, N))\geq 2$, as
required. 
\end{proof}

\begin{lemma}
\label{xxlem3.17} 
Let $A$ be a noetherian AS Gorenstein algebra.
Let $B$ be a graded subalgebra of $A$ such that $A_B$ and $_BA$ are finite. 
Then there is a duality
$$\add_{\grm \mhyphen B} A\cong  \add_{B\mhyphen \grm} A.$$
\end{lemma}

\begin{proof} Let $d$ be the injective dimension of $A$ and $\ell$ be the
AS index of $A$. We will show that the duality functor is 
$F:=R^d \Gamma_{\fm_B}(-)^*(\ell)$.  By Lemma \ref{xxlem3.11}(1), if
$M$ is a graded left $A$-module, then 
$R^d\Gamma_{\fm_A}(M)^*=R^d\Gamma_{\fm_B}(M)^*$. Hence
$$F({_B A})=R^d\Gamma_{\fm_A}(A)^*(\ell) \cong A_B,$$
where $A_B$ is viewed as a graded right $B$-module. Thus $F$ is a functor
from $\add_{B\mhyphen \grm} A$ to $\add_{\grm \mhyphen B} A$.
Similarly, $G:=R^d \Gamma_{\fm_{B^{\op}}}(-)^*(\ell)$ is a functor 
from $\add_{\grm \mhyphen B} A$ to $\add_{B\mhyphen \grm} A$.
It is clear that $GF(_BA)={_BA}$ and $FG(A_B)=A_B$. 
Therefore $F$ induces 
a duality $\add_{\grm \mhyphen B} A\cong  \add_{B\mhyphen \grm} A.$
\end{proof}

Theorem \ref{xxthmA}~now follows from 
Proposition \ref{xxpro2.3}, Corollary \ref{xxcor2.6} and Lemma 
\ref{xxlem3.17}.


\section{A noncommutative version of Herzog's theorem}
\label{xxsec4}

A theorem of Herzog states that if $G$ is a finite 
subgroup of ${\text{GL}}(2,\kk)$ acting linearly on $\kk[x_1,x_2]$, then 
the indecomposable MCM $\kk[x_1,x_2]^G$-modules are precisely the 
indecomposable $\kk[x_1,x_2]^G$-direct summands of $\kk[x_1,x_2]$ \cite{He} (or, see \cite[Section~6.1]{LW}). 
Here we establish a noncommutative version of this nice result.

\medbreak Consider the following terminology.

\begin{definition} 
\label{xxdef4.1}
Let $A$ be a noetherian ring. 
\begin{enumerate}
\item[(1)] 
A subring $B\subseteq A$ is called a \emph{splitting subring} of $A$ if 
$A_B$ and $_BA$ are finite and if
$B$ is a direct summand of $A$ as a $B$-bimodule.
\item[(2)] 
When $B$ is a splitting subring of $A$, the projection $\Phi:A\to B$ is a $B$-bimodule homomorphism called
the \emph{Reynolds operator}. 
\end{enumerate} 
\end{definition}

Note that as a consequence of (1), $B$ is noetherian. Consider the
following lemmas about splitting subrings.

\begin{lemma} 
\cite[Lemma 2.4]{KKZ2} 
\label{xxlem4.2} 
If $H$ is a semisimple Hopf algebra, with $A$ a noetherian 
left $H$-module algebra, then $A^H$ is a splitting subring 
of $A$, where the Reynolds operator $\Phi$ is given by the 
action of the integral of $H$ on $A$. \qed
\end{lemma}

\begin{lemma} 
\label{xxlem4.3} 
Let $A$ be a noetherian commonly graded algebra.
Suppose that $B$ is a splitting subring of $A$. 
\begin{enumerate} 
\item[(1)] 
We have that $\cd(B)=\cd(A)$. 
\item[(2)]
$B$ satisfies $\chi$ if and only if $A$ does. 
\item[(3)] 
$B$ has balanced dualizing complex if and only if $A$ does. 
\item[(4)] 
If $A$ is {\rm CM}, then so is $B$.
\item[(5)] 
If $A$ is {\rm CM}, then $A$ is {\rm MCM} as a left $B$-module. 
\end{enumerate} 
\end{lemma}

\begin{proof} 
(1) By Lemma~\ref{xxlem3.11}(1), $\cd(A)\leq \cd(B)$. Since
$R\Gamma_{\fm_A}(A)=R\Gamma_{\fm_B}(A)$,  again by
Lemma~\ref{xxlem3.11}(1), we have that
$R\Gamma_{\fm_B}(B)$ is a direct summand of  $R\Gamma_{\fm_A}(A)$. Therefore 
\begin{eqnarray*}
\max\{i\mid R^i\Gamma_{\fm_A}(A)\neq 0\} &\geq& \max\{i\mid R^i\Gamma_{\fm_B}(B)
\neq 0\}.
\end{eqnarray*}
 By Theorem \ref{xxthm3.8}(7), $\cd(A)\geq \cd(B)$, so the
assertion follows.

\smallskip

(2) We adapt the proof of \cite[Proposition 8.7]{AZ}. By \cite[Theorem
8.3(2)]{AZ}, if $B$ satisfies $\chi$, so does $A$. Conversely, we assume
that $A$ satisfies $\chi$. Let $M$ be a noetherian graded left
$B$-module. We claim that $M$ satisfies the $\chi$-condition, namely,
that $\uExt^i_{B}(S,M)$ is finite dimensional for all $i$, for all simple left
$B$-modules $S$. Let $N=A\uotimes_B M$. Since $_BM$ is finite, so is
$_AN$. Thus $_AN$ and $_BN$ are noetherian graded modules. Say 
$A$ decomposes as $B \oplus C$ as $B$-bimodules, then the Reynolds
operator $\Phi$ induces a decomposition of left $B$-modules $_BN=
{}_BM\oplus (_BC\otimes {}_B M)$. It suffices to show that $_BN$
satisfies the $\chi$-condition. By a change of rings spectral sequence,
\begin{eqnarray*}
E^{p,q}_{2}:=\uExt^p_A(\uTor_q^{B}(A, {}_BS), {}_AN) &\Rightarrow& \uExt^{p+q}_B(S, {}_BN),
\end{eqnarray*}
one sees that $_AN$ satisfies $\chi$ implies
that $_BN$ satisfies $\chi$.

\smallskip

(3) This follows from parts (1,2) and Theorem \ref{xxthm3.8}(2).

\smallskip

(4) The assertion follows from the fact that $R^i\Gamma_{\fm_B}(B)$ is
a direct summand of \linebreak $R^i\Gamma_{\fm_B}(A) \cong
R^i\Gamma_{\fm_A}(A)$.

\smallskip

(5) The assertion follows from  part~(1)  and
Lemma~\ref{xxlem3.11}(2).
\end{proof}

This brings us to the main result of this section.

\begin{theorem} 
\label{xxthm4.4} 
Let $A$ be a generalized AS regular algebra of global
dimension 2, and let $B$ be a splitting subring of $A$. Then, up to
isomorphism and a degree shift, the indecomposable {\rm MCM} left $B$-modules
are precisely the indecomposable $B$-direct summands of $_BA$.
\end{theorem}

An example of this result was discussed in \cite[Section~2.3]{Jo6}.

\begin{proof}[Proof of Theorem \ref{xxthm4.4}] 
Since $A$ is generalized AS regular, it is CM and MCM 
by Lemma~\ref{xxlem3.11}(5). Hence $A$ is left MCM  
over $B$ by Lemma~\ref{xxlem4.3}(5), and every $B$-direct summand of
$_BA$ is MCM over $B$, as well, by Lemma \ref{xxlem3.11}(3).

To begin the other direction, let $D_A$ be the balanced dualizing
complex over $A$, which exists by Theorem~\ref{xxthm3.8}(2).  Let
$d=\cd(A)$. Then $D_A\cong P[d]$, where $P=R^d\Gamma_{\fm_A}(A)^*$ by
Theorem \ref{xxthm3.8}(3). Let $D_B$ be the balanced dualizing complex
over $B$, which exists by Lemma~\ref{xxlem4.3}(3). It follows that $D_B$
has the form $\Omega[d]$, where $\Omega= R^d\Gamma_{\fm_B}(B)^*$ by
Theorem \ref{xxthm3.8}(3).

For every left MCM $B$-module $M$, define $M^\vee$ to be 
\[
\begin{array}{rll}
M^\vee &:=
\RHom_B(M[d],D_B)& =\RHom_B(M,D_B)[-d]\\
&=R^d\Gamma_{\fm_B}(M)^*
&=\Hom_B(M,\Omega),
\end{array}
\]
where the last two equalities hold
using local duality by Theorem~\ref{xxthm3.8}(6). Then 
\[
(M^\vee)^\vee ~\cong~ \Hom_{B^{\op}}(M^\vee[d],
D_B)~\cong~ \Hom_{B^{\op}}(\RHom_B(M,D_B), D_B) ~\cong~ M.
\]
Denote by $C$ the kernel of the Reynolds operator $\Phi:A\to B$. 
Applying $R^d\Gamma_{\fm_B}(-)^*$ to $\Phi$ gives a
split exact sequence of $B$-bimodules 
\begin{equation}
\label{E4.4.1}\tag{E4.4.1} 
0\to \Omega\to P\to C^{\vee}\to 0. 
\end{equation}
Applying $\Hom_{B^\op}(M^\vee,-)$ to \eqref{E4.4.1} gives 
a split exact sequence of left $B$-modules 
\[
0\to \Hom_{B^\op}(M^\vee,\Omega)\to
\Hom_{B^\op}(M^\vee,P)\to
\Hom_{B^\op}(M^\vee,C^{\vee})\to 0.
\] 
This shows that $M\cong \Hom_{B^\op}(M^\vee,\Omega)$
is a $B$-direct summand of the left $A$-module $\Hom_{B^{\op}}(M^\vee,P)$.

By the theory of dualizing complexes \cite[Proposition 3.4]{Ye}, $M^\vee$ is a 
finite right $B$-module. Since $P$ is invertible, the left $A$-module $_AP$ is a 
progenerator; so, $\dep(P)=\dep(A)=2$.  By Lemma
\ref{xxlem3.16}, $Y:=\Hom_{B^{\op}}(M^\vee,P)$ is a noetherian
left $A$-module of depth at least~2. But the maximal depth of $Y$ is 2, as $A$
has global dimension~2. Therefore, by definition, $Y$ is  MCM over $A$. Since
$A$ is generalized AS regular of global dimension~2,  
$Y$ is projective as a left $A$-module
by Lemma \ref{xxlem3.13}. Since $M\cong
\Hom_{B^{\op}}(M^\vee,\Omega)$ is a $B$-direct summand of $Y$, 
$M$ is a direct summand of a projective left
$A$-module. If $M$ is indecomposable, then, up to a degree shift, it is a
direct summand of $_BA$. 
\end{proof}

\begin{corollary} 
\label{xxcor4.5}
Let $H$ be a semisimple Hopf algebra. Suppose that $A$ is an $H$-module 
algebra satisfying the hypotheses of Theorem~\ref{xxthm4.4}. Then,
up to isomorphism and a degree shift, the indecomposable {\rm MCM} 
left $A^{H}$-modules are precisely the
indecomposable direct summands of $A$ as a left $A^{H}$-module.
\end{corollary}

\begin{proof} By Lemma~\ref{xxlem4.2}, take $A^H$ to be the splitting
subring of $A$. Now apply Theorem~\ref{xxthm4.4}. 
\end{proof}

Theorem \ref{xxthmC} follows immediately from Corollary \ref{xxcor4.5}.

\section{A noncommutative version of Leuschke's theorems} 
\label{xxsec5}

In this section we prove noncommutative versions, 
Theorem~\ref{xxthmD} (Theorem~\ref{xxthm5.4})  and Theorem~\ref{xxthmE} (Theorem~\ref{xxthm5.7}), of 
results of Leuschke \cite[Theorem~6 and Proposition~8]{Le1}. 
These results pertain to the global dimension and the 
{\it representation dimension} (Definition~\ref{xxdef5.5}) of 
endomorphism rings of MCM modules.  

\medbreak We begin with the following preliminary result.
Recall that in the notation $\End_{B}(M)$ or $\uEnd_{B}(M)$,
the module $M$ is considered as a right $B$-module. 

\begin{lemma}
\label{xxlem5.1} 
Let $B$ be an algebra that is module-finite over a central 
noetherian subalgebra $C$, and let $M$ be a finite right 
$B$-module. Then $\End_{B}(M)$ is  a noetherian PI algebra. 
If further, $B$ is commonly graded and $M$ is a graded right
$B$-module, then $\uEnd_{B}(M)$ has a balanced dualizing complex.
\end{lemma}

\begin{proof} Note that $C$ is a noetherian commutative algebra.
Since $C$ is central in $B$ there is a natural map from $C$
to $\End_{B}(M)$. Since $B$ is finite over $C$, so is $M$.
Applying $\Hom_{B^{\op}}(-,M)$ to a surjective map $B^{\oplus n}
\to M$, one sees that $\End_{B}(M)$ is a $C$-submodule of
$\Hom_{B^{\op}}(B^{\oplus n},M)\cong M^{\oplus n}$, 
which is finite over $C$. Hence $\End_{B}(M)$ is finite
over $C$. Therefore $\End_{B}(M)$ is a noetherian PI ring. 

Since $B$ and $M$ are graded, so is $\uEnd_{B}(M)$. 
A graded version of \cite[Corollary~0.2]{WZ1} implies that
$\uEnd_{B}(M)$ has a balanced dualizing complex.
\end{proof}

Now the main results of this section hold for algebras of 
{\it finite CM type}, which we define next.

\begin{definition} 
\label{xxdef5.2}
We say that a  (left) noetherian commonly graded algebra $A$ is 
of {\it finite Cohen-Macaulay \textnormal{(}finite CM\textnormal{)} type} 
if $A$ has, up to a degree shift, finitely many non-isomorphic 
indecomposable MCM modules.
\end{definition}

\begin{example} 
\label{xxex5.3}
Let $A$ be a connected graded or generalized AS regular algebra 
of dimension 2 that admits an action of a semisimple Hopf 
algebra $H$. Then the fixed subring $B:=A^H$ is of finite CM 
type by Corollary~\ref{xxcor4.5}. Since the category
of finitely generated graded  left $B$-modules is a 
Krull-Schmidt category, the $B$-module $A$ 
has only finitely many direct summands. 
\end{example}

Next we prove Theorem \ref{xxthmD}.
\begin{theorem}
\label{xxthm5.4} 
Let $B$ be a noetherian commonly graded 
algebra such that $B$ has a balanced dualizing complex. Suppose that
$B$ is {\rm CM} and of finite CM type.
Let $M$ be a finite direct sum of {\rm MCM} right $B$-modules that contain
at least one copy of each {\rm MCM} module, up to a degree shift. Let 
$E:=\End_{B}(M)$.
\begin{itemize}
\item[(1)] If $E$ is right noetherian, then $E$ has right global
dimension at most $\max\{2,d\}$, where $d=\cd(B)$.
\item[(2)] If $d\geq 2$, then the right global dimension of $E$ is equal to $d$.
\end{itemize}
\end{theorem}

\begin{proof} (1) We modify the proof given in \cite[Theorem 6]{Le1}.
The case $d=1$ is easy, so we assume that $d\geq 2$. 

Let $X=X_E$ be a finite right $E$-module, and consider the first $d-1$
steps in the projective resolution of $X_E$:
\begin{equation}
\label{E5.4.1}\tag{E5.4.1}
P_{\bullet}:\quad 
P_{d-1}\xrightarrow{f_{d-1}}P_{d-2}\to \cdots \to P_{1}
\xrightarrow{f_1} P_0
\end{equation}
with $X={\text{coker}}\; f_1$. Since $E$ is right noetherian by assumption,
we may assume that each $P_i$ is a finite projective right
$E$-module. By Proposition \ref{xxpro2.5}(3),  \linebreak
$P_i=\uHom_{\grm \mhyphen B^{\op}}(M,N_i)$
for some $N_i\in \add_{\grm \mhyphen  B^{\op}} M$. By Lemma~\ref{xxlem3.11}(3)
each $N_i$ is a MCM right $B$-module. By Proposition \ref{xxpro2.5}(3),
$f_i=\uHom_{\grm\mhyphen B^{\op}}(M,g_i)$, where 
$g_i\in \uHom_{\grm \mhyphen B^{\op}}(N_i,N_{i-1})$.
Putting these facts together, we have the following sequence of morphisms between
MCM graded right $B$-modules:
\begin{equation*}
N_{\bullet}:\quad 
N_{d-1}\xrightarrow{g_{d-1}}N_{d-2}\to \cdots \to N_{1}
\xrightarrow{g_1} N_0
\end{equation*}
such that $\uHom_{\grm \mhyphen B^{\op}}(M,N_{\bullet})=P_{\bullet}$. 

Since $B$ is CM, every indecomposable direct summand of 
$B$ is an MCM module, which must be a direct summand of $M$. 
Hence $M=M_0\oplus M_1$, where
$M_0$ is  a graded progenerator of $B$.  Therefore $P_{\bullet}=
\uHom_{\grm\mhyphen B^{\op}}(M,N_{\bullet})$
contains $\uHom_{\grm\mhyphen B^{\op}}(M_0,N_{\bullet})$ as a 
direct summand. Since $\uHom_{\grm\mhyphen B^{\op}}(M,N_{\bullet})$
is exact, so is $\uHom_{\grm\mhyphen B^{\op}}(M_0,N_{\bullet})$. 
Note that $\uHom_{\grm\mhyphen B^{\op}}(M_0,-)$ is an equivalence of 
categories by Morita theory.  Therefore, $N_{\bullet}$ is exact. 
Let $N_d$ be the kernel of $g_{d-1}$. Then we have an exact 
sequence
\begin{equation}
\label{E5.4.2}\tag{E5.4.2}
0\to N_d\to N_{d-1}\xrightarrow{g_{d-1}}N_{d-2}\to \cdots \to N_{1}
\xrightarrow{g_1} N_0.
\end{equation}
By Lemma \ref{xxlem3.15}, $N_d$ is a MCM. Because $M$ contains all 
indecomposable MCM modules,  $N_d$ is in 
$\add_{\grm\mhyphen  B^{\op}} M$. Applying
$\uHom_{\grm\mhyphen B^{\op}}(M,-)$ to \eqref{E5.4.2}, one has an exact sequence
$$0\to \uHom_{\grm\mhyphen B^{\op}}(M,N_d)\to  
P_{d-1}\xrightarrow{f_{d-1}}P_{d-2}\to \cdots \to P_{1}
\xrightarrow{f_1} P_0\to X\to 0.$$
Since $N_d\in \add_{\grm\mhyphen  B^{\op}} M$, 
$\uHom_{\grm\mhyphen  B^{\op}}(M,N_d)\in \grprm \mhyphen E$ by 
Proposition~\ref{xxpro2.5}(2). Therefore $X$ has projective
dimension at most $d$. This shows that $E$ has right global
dimension at most $d$. 

Next we show that the right global dimension of $E$ is $d$ if $d\geq 2$. 
Let $Y$ be 
a nonzero finite dimensional right $B$-module (so, $\dep(Y) =0$), and consider an exact sequence
$$N_1\xrightarrow{g_1}N_0\xrightarrow{g_0} Y\to 0,$$ where $N_0$ and 
$N_1$ are 
projective $B$-modules.
Let $X$ be the cokernel of the map $\uHom_{\grm\mhyphen B^{\op}}(M,g_1)$,  so we have a 
partial projective resolution $P_1\to P_0\to X$, where 
$P_i=\uHom_{\grm\mhyphen B^{\op}}(M,N_i)$.
Extend this to a projective resolution \eqref{E5.4.1}. 
We claim that the projective dimension
of $X$ is $d$. If not, we may assume that $P_{d}=0$. 
By the construction, as before, we have an exact sequence 
$N_{\bullet}$ with $N_d=0$. 
By Lemma~\ref{xxlem3.15}, $\dep(Y)\geq \dep(N_{d-1})
-(d-1)>0$, which contradicts  the fact that $\dep(Y)=0$.

\medbreak (2) This follows exactly as in the proof of \cite[Theorem~6]{Le1}.
\end{proof}

\begin{definition}
\label{xxdef5.5}
Let $A$ be a noetherian CM commonly graded algebra of depth
$d$ and  let $\omega=R^d\Gamma_{\fm}(A)^*$. The \emph{representation
dimension} of $A$ is defined to be
$$\repdim(A)=\inf \{\gldim\left(\End_{A}(A\oplus 
\omega\oplus M)\right) ~:~ M \text{ is a MCM left $A$-module}\}.$$
\end{definition}

We obtain the following generalization of a result of Leuschke 
\cite[Proposition~8]{Le1}  as one immediate consequence 
of Theorem \ref{xxthm5.4}.

\begin{corollary}
\label{xxcor5.6} Let $B$ be a noetherian, {\rm CM},
commonly graded algebra, of finite {\rm CM} type, that is 
module-finite over a central noetherian graded subalgebra. Then we 
have that $\repdim(B) \leq \max\{2,\cd(B)\}$.
\end{corollary}

\begin{proof}
Let $E:= \End_{B}(B\oplus R^{\cd(B)}\Gamma_{\mathfrak m}(B)^* \oplus N)$, 
where $N$ is the direct sum of the remaining indecomposable MCM 
left $B$-modules. Then $E$ is noetherian by Lemma~\ref{xxlem5.1}, and 
hence, $E$ has global dimension at most $\max\{2,\cd(B)\}$ by Theorem~\ref{xxthm5.4}.
\end{proof}

There are stronger results for graded CM algebras $B$, of finite CM type, 
when $\cd(B) = 2$.   We begin with a proof of Theorem \ref{xxthmE}, 
which is an easy consequence of Theorem \ref{xxthm5.4}.

\begin{theorem}
\label{xxthm5.7} Let $B$ be a noetherian commonly graded 
algebra that is {\rm CM} with $\cd(B)=2$, of finite CM type, and 
module-finite over a central noetherian graded subalgebra $C$.
Let $M$ be a finite direct sum of {\rm MCM} right $B$-modules that contains
at least one copy of each {\rm MCM} module, up to a degree shift. 
Then $E:=\uEnd_{B}(M)$ is a noetherian PI generalized AS 
regular algebra of dimension $2$.
\end{theorem}

\begin{proof} By Theorem \ref{xxthm5.4}, $E$ has right global
dimension 2. By Lemma \ref{xxlem5.1}, $E$ is noetherian PI, equipped
with a balanced dualizing complex. Hence, $E$ has left global
dimension 2. By Lemma \ref{xxlem3.16},
the depth (and thus cohomological dimension) of $E$ is 2, 
or equivalently, $R^i\Gamma_{\fm}(E)=0$ for $i=0,1$.

Let $\Omega$ be $R^2\Gamma_{\fm}(E)^*$. It remains to show
that $\Omega$ is invertible. By Theorem \ref{xxthm3.8}(3),
$D:=\Omega[2]$ is the balanced dualizing complex over $E$.
Then 
$$E=\RHom_{E}(D,D)\cong R\Gamma_{\fm}(D)
\cong R\Gamma_{\fm}(\Omega)^*[-2].$$
Therefore $R^i\Gamma_{\fm}(\Omega)^*=0$ for $i=0,1$,
and $R^2\Gamma_{\fm}(\Omega)^*=E$. Thus
$\Omega$ is MCM. Since $E$ is MCM, it is depth-homogeneous 
in the sense of \cite[page 521]{WZ3}. By \cite[Theorem~3.4]{WZ3},
$\Omega$ is a projective (left and right) $E$-module.
By the discussion before the proof of \cite[Theorem~0.2, page 561]{YZ},
since $E$ is Gorenstein, $D$ is also a two-sided tilting 
complex over $E$, or equivalently, an invertible complex over $E$ by 
\cite[Definition 2.1]{YZ}. Since $D$ is a stalk complex $\Omega[2]$ 
that is invertible, we obtain that $\Omega$ is an invertible 
$E$-bimodule. Therefore $E$ is generalized AS regular.
\end{proof}

\begin{corollary}
\label{xxcor5.8}
Let $B$ be a noetherian commonly graded 
algebra that is {\rm CM} with $\cd(B)=2$, of finite CM type, 
and module-finite over a central noetherian graded subalgebra $C$. 
Then the center $Z$ of $B$ is reduced.
\end{corollary}

\begin{proof} Let $M$ be a finite direct sum of all MCM right $B$-modules 
that contain at least one copy of each MCM module, up to a degree shift,
and a copy of $B$. Then the central action of $Z$ on $M$ is 
faithful. Let $E:=\End_{B^{\op}}(M)$. By Theorem \ref{xxthm5.7}, $E$ 
is generalized AS regular and is PI. By 
\cite[Theorem~1.4(1)]{StZ} (which holds for 
generalized AS regular algebras), $E$ is semiprime. Hence, any
central subring of $E$, including $Z$, is reduced.
\end{proof}

\begin{remark} \label{xxrem5.9}
Theorem~\ref{xxthm5.7} and Corollary~\ref{xxcor5.8} apply 
to the fixed subring 
$B = A^H$ arising from an action of a semisimple Hopf algebra  
$H$ on an AS regular algebra $A$ of dimension 2; see Example~\ref{xxex5.3}.
\end{remark}

Theorem \ref{xxthm5.7} fails for $\cd(B)\geq 3$, even in the
commutative case, as the following example demonstrates.

\begin{example}
\label{xxex5.10}
\cite[Example 11]{Le1}
Let $B$ be the subring of $A=\kk[x,y,z]$ generated by 
degree two elements,
namely, $B=\kk[x^2,xy,y^2,xz,yz,z^2]$. 
(The subring $B$ arises as the fixed subring  of $A$ 
under the $\mathbb{Z}_2$-action given by negation.) 
Then $B$ has cohomological dimension~3 
(e.g. by Lemma~\ref{xxlem3.15}), and $B$ has 
three MCM graded 
modules, up to degree shift and isomorphism. 
When $M$ is the direct
sum of these three MCM modules, we obtain 
that $E:=\uEnd_B(M)$ has global dimension 3. 
However, $E$ also has depth~2  (which can be 
checked with Macaulay 2, as mentioned in \cite{Le1}).  
Therefore $E$ is not CM, and hence not AS regular by 
Lemma~\ref{xxlem3.11}(4).
\end{example}

\subsection*{Acknowledgments}

The authors would like to thank Graham Leuschke 
for many useful conversations on the subject and thank the referee 
for his/her very careful reading and extremely valuable comments.
C. Walton and J.J. Zhang were supported by the US National Science
Foundation: NSF grants DMS-1550306 and both DMS-0855743 and DMS-1402863
respectively. E. Kirkman was supported by Simons Grant \#208314.\\\\

\bibliography{hopfmckayII}

\end{document}